\documentclass[12pt]{amsart}
\usepackage{amssymb, amsmath, amsfonts, amsthm, esint,graphicx}
\usepackage[abbrev]{amsrefs}
\usepackage[margin=1 in]{geometry}
\usepackage[colorlinks=true, linkcolor=blue]{hyperref}

\def \dd {\partial}

\DeclareMathOperator{\Ric}{Ric}

\DeclareMathOperator{\Hess}{Hess}

\DeclareMathOperator{\im}{im}

\title{The first nonzero eigenvalue of the $p$-Laplacian on differential forms}
\author{Shoo Seto}
\date{\today}
\address{Department of Mathematics\\
         University of California\\
         Irvine, CA 92697}
\email{\href{mailto:shoos@uci.edu}{shoos@uci.edu}}
\keywords{$p$-Laplacian, Hodge-Laplacian, Weitzenb\"ock curvature}

\theoremstyle{definition} 
\newtheorem{definition}{Definition}[section]

\newtheorem{lemma}{Lemma}[section]
\newtheorem{remark}{Remark}[section]

\newtheorem{proposition}{Proposition}[section]
\newtheorem{theorem}{Theorem}[section]

\begin{document}
\maketitle
\begin{abstract}
We introduce a generalization of the $p$-Laplace operator to act on differential forms and generalize an estimate of Gallot-Meyer \cite{gallot-meyer} for the first nonzero eigenvalue on closed Riemannian manifolds.
\end{abstract}
\section{Introduction}
Let $(M,g)$ be an $n$-dimensional closed Riemannian manifold.   Motivated from the variational characterization of the Laplacian eigenvalue problem, we define the $L^p$-Dirichlet integral on $k$-forms (introduced in \cite{scott}) by
\begin{equation}\label{Lpdirichlet}
\mathcal{F}[\alpha] := \int_M \|d\alpha\|^p + \|d^*\alpha\|^p, \quad \alpha \in \Omega^k(M),
\end{equation}
where $d^*$ is the $L^2$-adjoint of the exterior derivative $d$.  Note that $\mathcal{F}[\alpha] = 0$ if and only if $\alpha \in \mathcal{H}^k(M)$, that is, the minimum is zero and is attained for harmonic $k$-forms, i.e. $\alpha \in \ker(d)\cap \ker(d^*)$.   For a nonzero infimum we consider the space
\begin{equation}\label{orthoharmo}
A_k := \left\{ \alpha \in \mathcal{W}^{1,p}(\Omega^k(M)) \ | \ \fint_M \|\alpha \|^p= 1, \int_M\|\alpha\|^{p-2}\langle \alpha,\omega\rangle = 0, \omega \in \mathcal{H}^k(M) \right\},
\end{equation}
where the space $\mathcal{W}^{1,p}(\Omega^k(M))$ is the $(1,p)$-Sobolev space of differential $k$-forms defined in \cite{scott}.  See \S \ref{variation} for the precise definition.
Computing the Euler-Lagrange equation leads us to the defining the following operator
\begin{definition}[$p$-Hodge Laplacian]
\begin{equation}\label{p-laplace}
\Delta_{p}\alpha := d^*(\|d\alpha\|^{p-2}d\alpha) + d(\|d^*\alpha\|^{p-2}d^*\alpha), \quad \alpha \in \Omega^k(M).
\end{equation}
\end{definition}
When $p=2$, this becomes the usual Hodge Laplacian.  For $p\neq 2$ and $\alpha \in C^\infty(M)$ $\Delta_p$ becomes the usual $p$-Laplacian.  The corresponding eigenvalue equation is given by
\begin{equation}\label{eigenform}
\Delta_{p}\alpha = \lambda \|\alpha\|^{p-2}\alpha, \quad \alpha \in \Omega^k(M)
\end{equation}
and the variational principle tells us that
\begin{equation*}
\lambda_1 = \inf\{\mathcal{F}[\alpha] \ | \ \alpha \in A_k\}.
\end{equation*}
See \S \ref{variation} for details. When $p=2$, there is much work on the spectrum of the Hodge-Laplacian acting on differential forms.  Among many others, we point out the work of Gallot-Meyer \cite{gallot-meyer}, \cite{gallot-meyer-2} who show an estimate of the first eigenvalue using bounds from the Weitzenb\"ock curvature on compact Riemannian manifolds.   For manifolds with boundary, among many others, see works of Kwong \cite{kwong}, Savo \cite{savo}, Raulot-Savo \cite{raulot-savo}, and references therein.

For $p\neq 2$, the $p$-Laplace eigenvalue problem on $0$-forms (functions) has attracted much attention.  See notes by Lindqvist \cite{lindqvist} for a general reference on the $p$-Laplace equation.  For estimates on the first eigenvalue relating to the curvature, among many other works, see Matei \cite{matei}, Naber-Valtorta \cite{naber-valtorta}, Seto-Wei \cite{seto-wei} for eigenvalue estimates with $\Ric \geq K$, $K\in \mathbb{R}$.

In this paper we prove the following lower bound estimate for the first eigenvalue 
\begin{theorem}\label{mainthm}
Let $M^n$ be a closed Riemannian manifold with the eigenvalues of the curvature operator bounded below by $H\in \mathbb{R}$ and $p\geq 2$.  Then
\begin{equation*} 
\lambda_1 \geq \left(\frac{k(n-k)}{2^{\frac{2}{p}-1}\left(C+\frac{(p-2)}{2}\right)}H \right)^{\frac{p}{2}},
\end{equation*}
where 
\begin{equation*}
C=\max\left\{\frac{k}{k+1}, \frac{n-k}{n-k+1} \right\}.
\end{equation*}
\end{theorem}
\begin{remark}
When $p=2$, the above recovers the estimate due to Gallot-Meyer \cite{gallot-meyer} (see also \cite{gallot-meyer-2}), for $1 \leq k \leq \frac{n}{2}$,
\begin{equation*}
\lambda_1 \geq k(n-k+1)H.
\end{equation*}
\end{remark}

The organization of this paper is as follows.  In \S \ref{sec2} we review some known estimates for differential $k$-forms.  In \S\ref{variation} we show that the infimum can be characterized as an eigenvalue problem.  In \S\ref{sec4} we give the main estimate.  In \S\ref{sec5} we give a brief discussion on boundary conditions for differential forms and possible future directions.

\subsection*{Acknowledgments:} The author would like to thank Prof. Zhiqin Lu for very helpful discussions, suggestions and constant support, and to Prof. Guofang Wei and Prof. Qi Zhang for discussions and encouragement on the problem.  We also thank Prof. Nguyen Thac Dung for letting of know of an improvement on the constant, see Lemma \ref{lemma1}.

\section{Some estimates on $\Omega^k(M)$}\label{sec2}
We first recall the Weitzenb\"ock curvature
\begin{definition}
Let $p \in M$ and let $\{E_i\}_{i=1}^n$ be an orthonormal frame at $p$.  Then for $\alpha \in \Omega^k(M)$, define the Weitzenb\"ock curvature $W_k$ by
\begin{equation*}
W_k(\alpha)(X_1,\ldots,X_k) := \sum (R(E_j,X_i)\alpha)(X_1,\ldots,E_j,\ldots,X_k).
\end{equation*}
Note that on 1-forms, this is simply the Ricci tensor.
\end{definition}
If the eigenvalues of the curvature operator are bounded by $H\in \mathbb{R}$, we can show that
\begin{equation}\label{weitzenbocklower}
(W_k(\alpha),\alpha) \geq k(n+1-k)H\|\alpha\|^2.
\end{equation}
The Weitzenb\"ock curvature shows up in the main tool we will use in obtaining our estimate is the Bochner-Weitzenb\"ock formula for $k$-forms
\begin{equation}\label{bochner-weitzenbock}
\frac{1}{2}\Delta\|\alpha\|^2=(\Delta\alpha,\alpha)-\|\nabla\alpha\|^2-(W_k(\alpha),\alpha),
\end{equation}
where $\Delta:=\Delta_2 = dd^*+d^*d$.  Note that for exact 1-form $\alpha = df$, since $\nabla d f = \Hess f$, the usual Cauchy-Schwarz inequality will give us an estimate on the middle term.  For $k$-forms, we will need the following proved by Gallot-Meyer
\begin{lemma}[\cite{gallot-meyer}]
Let $\alpha \in \Omega^k(M)$, $1 \leq k \leq n-1$.  Then
\begin{equation}\label{twistor}
\|\nabla \alpha\|^2 \geq \frac{1}{k+1}\|d\alpha\|^2+\frac{1}{n-k+1}\|d^*\alpha\|^2.
\end{equation}
\end{lemma}
We give a proof for completeness.  The proof we give is in the context of conformal Killing forms and can be found in various sources, for instance, \cite{moroianu-semmelmann}.
\begin{proof}
Consider the two linear maps
\begin{align*}
&\iota:TM\otimes \Omega^k(M) \to \Omega^{k-1}(M)\\
&\iota(v,\alpha) = \iota_v\alpha
\end{align*}
and
\begin{align*}
&\wedge:\Omega^1(M)\otimes \Omega^{k}(M) \to \Omega^{k+1}(M)\\
&\wedge(\beta,\alpha)=\beta \wedge \alpha.
\end{align*}
Let $\iota^*$ and $\wedge^*$ be their metric adjoint.  Then
\begin{align*}
\wedge \circ \iota^*(\alpha)= 0  \text{ and } \iota\circ\wedge^*(\alpha) = 0,
\end{align*}
so that we get the decomposition
\begin{align*}
TM\otimes \Omega^k(M) \simeq \im(\iota^*)\oplus \im(\wedge^*)\oplus Y
\end{align*}
where $Y$ is the orthogonal complement.  By direct computation, we have for $\alpha \in \Omega^k(M)$,
\begin{align*}
\iota\circ \iota^*(\alpha) &= (n-k+1)\alpha \text{ and } \wedge\circ \wedge^*(\alpha) = (k+1)\alpha.
\end{align*}
Viewing $\nabla\alpha \in \Gamma(TM\otimes \Omega^k(M))$, From the decomposition,
\begin{align*}
\nabla \alpha = \iota^* \beta + \wedge^*\gamma + \delta,
\end{align*}
applying $\iota$, we have
\begin{align*}
\iota \nabla\alpha = (n-k+1)\beta.
\end{align*}
So the projection operator onto $\im(\iota^*)$ is given by
\begin{align*}
\pi_{\iota^*}\nabla \alpha= \frac{1}{n-k+1}\iota^*\iota \nabla \alpha
\end{align*}
and similarly
\begin{align*}
\pi_{\wedge^*}\nabla\alpha = \frac{1}{k+1}\wedge^*\wedge \nabla\alpha.
\end{align*}
Let $T\alpha := \pi_{T}\alpha$ the projection onto the orthogonal complement space.  Since
\begin{align*}
d\alpha = \wedge (\nabla \alpha) \text{ and } d^*\alpha=-\iota(\nabla \alpha),
\end{align*}
we have the decomposition
\begin{align*}
T\alpha (X) = \nabla_X\alpha-\frac{1}{k+1}\iota_Xd\alpha + \frac{1}{n-k+1}X^*\wedge d^*\alpha
\end{align*}
and taking the norm gives us
\begin{align*}
\|\nabla \alpha\|^2 = \|T\alpha\|^2+\frac{1}{k+1}\|d\alpha\|^2+\frac{1}{n-k+1}\|d^*\alpha\|^2,
\end{align*}
which implies \eqref{twistor}.
\end{proof}
\begin{remark}
The projection operator $T$ defined above is called the twistor operator and a form $\alpha \in \Omega^k(M)$ is called a conformal Killing form if $T\alpha = 0$.
\end{remark}
The following lemma was pointed out by N.T. Dung and gives us a way to control the interior product by using an orthogonal decomposition of forms as the image under an interior product.  
\begin{lemma}[Lemma 3.5 \cite{dung}]\label{lemma1}
Let $V \in TM$, $\alpha \in \Omega^{k+1}$, $\beta \in \Omega^k$.  Then
\begin{equation*}
|\langle \iota_V\alpha,\beta\rangle| \leq \|V\|\|\alpha\|\|\beta\|.
\end{equation*}
\end{lemma}

\section{Variational characterization of the eigenvalue}\label{variation}
In this section we will compute the Euler-Lagrange equation of \eqref{Lpdirichlet} and show that the extremal problem can be reformulated as an eigenvalue problem.  Analogous to the $0$-form (function) case, we will look at weak solutions lying the $(1,p)$-Sobolev space of differential $k$-forms first defined by Scott in \cite{scott} as
\begin{equation*}
\mathcal{W}^{1,p}(\Omega^k(M)) := \left\{  \alpha \in W(\Omega^k(M)) \ | \ \alpha, d\alpha, d^*\alpha \in L^p(\Omega^{*}(M))  \right\}
\end{equation*}
where $W(\Omega^k(M))$ is the classical Sobolev space of $k$-forms, i.e., $\alpha$ is locally integrable and admits a generalized gradient.

\begin{definition}
We say that $\lambda$ is an eigenvalue, if there exists a $k$-form $\alpha \in \mathcal{W}^{1,p}(\Omega^k(M))$ such that
\begin{equation*}
\int_M \|d\alpha\|^{p-2}\langle d\alpha,d\beta\rangle + \int_M\|d^*\alpha\|^{p-2}\langle d^*\alpha,d^*\beta\rangle = \lambda \int_M \|\alpha\|^{p-2}\langle \alpha,\beta\rangle,
\end{equation*}
for any $\beta \in C^\infty(\Omega^k(M))$.
\end{definition}
We will show the first nonzero eigenvalue $\lambda_1$ can be characterized as the infimum of the $L^p$-Dirichlet energy over the space $A_k$ given in \eqref{orthoharmo}.
\begin{proposition}
For closed manifolds $M$ and $p\geq 2$, 
\begin{equation*}
\lambda_1 = \inf \left\{\int_M \|d\alpha\|^p + \|d^*\alpha\|^p \ | \ \alpha \in A_k \right\}.
\end{equation*}
\end{proposition}
\begin{proof}
Let $\omega$ be a fixed harmonic form and let $\beta(t) \in  A$ for small $t>0$ such that $\beta(0) = \alpha$.  Computing the first variation of \eqref{Lpdirichlet}, we have
\begin{align*}
\frac{d}{dt}\mathcal{F}[\beta(t)]\biggr|_{t=0} &= p\int_M \|d\alpha\|^{p-2}\langle d\alpha,d\beta'(0)\rangle + \|d^*\alpha\|^{p-2}\langle d^*\alpha,d^*\beta'(0)\rangle\\
&=p\int_M \langle \Delta_p\alpha,\beta'(0)\rangle .
\end{align*}
Next we compute the variation of the constraints so that
\begin{align*}
\frac{d}{dt}\int_M \|\beta\|^p\biggr|_{t=0} &= p\int_M |\alpha|^{p-2}\langle \alpha,\beta'(0)\rangle
\end{align*}
and
\begin{align*}
\frac{d}{dt}\int_M \|\beta\|^{p-2}\langle\beta,\omega\rangle\biggr|_{t=0} &= (p-2)\int_M \|\alpha\|^{p-4}\langle \alpha,\beta'(0)\rangle\langle \alpha,\omega\rangle + \|\alpha\|^{p-2}\langle \beta'(0),\omega\rangle.
\end{align*}
By Lagrange multiplier method, there must be some $\lambda$ and $\mu$ such that for $\beta \in \Omega^k(M)$, 
\begin{align*}
\int_M \langle \Delta_p\alpha,\beta\rangle = \lambda \int_M\|\alpha\|^{p-2}\langle \alpha,\beta\rangle+\mu\int_M\|\alpha\|^{p-4}\langle \alpha,\beta\rangle\langle\alpha,\omega\rangle +\|\alpha\|^{p-2}\langle\beta,\omega\rangle.
\end{align*}
Setting $\beta = \omega$, we have
\begin{align*}
0=\mu\int_M\|\alpha\|^{p-4}\langle\alpha,\omega\rangle^2 +\|\alpha\|^{p-2}\|\omega\|^2 
\end{align*}
so that $\mu = 0$.  Therefore,
\begin{equation*}
\Delta_p\alpha = \lambda\|\alpha\|^{p-2}\alpha.
\end{equation*}
\end{proof}

\section{Proof of theorem \ref{mainthm}}\label{sec4}
We will consider the following integral
\begin{align*}
\int_M \langle \Delta_p\alpha,\Delta\alpha\rangle = \int_M\langle \Delta_p\alpha,dd^*\alpha\rangle + \int_M\langle \Delta_p,d^*d\alpha\rangle.
\end{align*} 
Let $\alpha \in \Omega^k(M)$ be an eigenform satisfying \eqref{eigenform}.  Then
\begin{align}
\begin{split}\label{eq1}
\int_M \langle \Delta_{p}\alpha,d^*d\alpha\rangle &=\lambda \int_M\|\alpha\|^{p-2}\langle \alpha, d^*d\alpha\rangle\\
&=\lambda \int_M \langle d(\|\alpha\|^{p-2}\alpha),d\alpha\rangle \\
&=\lambda\int_M\langle d(\|\alpha \|^{p-2})\wedge \alpha,d\alpha\rangle + \lambda\int_M\|\alpha\|^{p-2}\|d\alpha\|^2
\end{split}
\end{align}
and
\begin{align}
\begin{split}\label{eq2}
\int_M\langle \Delta_{p}\alpha,dd^*\alpha\rangle &= \lambda\int_M \|\alpha\|^{p-2}\langle \alpha,dd^*\alpha\rangle \\
&=\lambda \int_M \langle d^*(\|\alpha\|^{p-2}\alpha),d^*\alpha\rangle \\
&=\lambda \int_M \|\alpha\|^{p-2}\|d^*\alpha\|^2 - \lambda\int_M\langle \iota_{\nabla \|\alpha\|^{p-2}}\alpha,d^*\alpha\rangle.
\end{split}
\end{align}
On the other hand, by using the Bochner-Weitzenb\"ock formula \eqref{bochner-weitzenbock} we have
\begin{align}
\begin{split}\label{eq3}
\int_M \langle \Delta_{p}\alpha,\Delta\alpha\rangle &= \lambda\int_M\|\alpha\|^{p-2}\langle \alpha, \Delta \alpha\rangle \\
&=\lambda \int_M \left((p-2)\|\alpha\|^{p-2}|\nabla\|\alpha\||^2+\|\alpha\|^{p-2}\|\nabla \alpha\|^2+\|\alpha\|^{p-2}(W_k(\alpha),\alpha)\right).
\end{split}
\end{align}
Combining \eqref{eq1}, \eqref{eq2}, and \eqref{eq3}, we obtain
\begin{align}
\begin{split}\label{maineq}
& \int_M \langle d(\|\alpha\|^{p-2})\wedge \alpha,d\alpha \rangle - \int_M\langle\iota_{\nabla\|\alpha\|^{p-2}}\alpha,d^*\alpha\rangle +  \int_M \|\alpha\|^{p-2}\|d\alpha\|^2 + \int_M \|\alpha\|^{p-2}\|d^*\alpha\|^2\\
&=\int_M \left((p-2)\|\alpha\|^{p-2}|\nabla\|\alpha\||^2+\|\alpha\|^{p-2}\|\nabla \alpha\|^2+\|\alpha\|^{p-2}(W_k(\alpha),\alpha)\right).
\end{split}
\end{align}
Using Lemma \ref{lemma1}, the first term of \eqref{maineq} can be estimated as
\begin{align*}
\int_M \langle d(\|\alpha\|^{p-2})\wedge \alpha,d\alpha\rangle &= \int_M \langle \alpha,\iota_{\nabla\|\alpha\|^{p-2}}(d\alpha)\rangle \\
&\leq \int_M\|\nabla \|\alpha\|^{p-2}\|\|d\alpha\|\|\alpha\| \\
&=(p-2)\int_M \|\alpha\|^{\frac{p-2}{2}}\|\nabla\|\alpha\|\|\|\alpha\|^{\frac{p-2}{2}}\|d\alpha\|\\
&\leq \frac{(p-2)}{2}\int_M\|\alpha\|^{p-2}\|\nabla\|\alpha\|\|^2 + \frac{(p-2)}{2}\int_M \|\alpha\|^{p-2}\|d\alpha\|^2
\end{align*}
and similarly for the second term,
\begin{align*}
-\int_M\langle \iota_{\nabla\|\alpha\|^{p-2}}\alpha,d^*\alpha\rangle &\leq  \int_M \|\nabla\|\alpha\|^{p-2}\|\alpha\|\|d^*\alpha\| \\
&= (p-2)\int_M\|\alpha\|^{\frac{p-2}{2}}\|\nabla\|\alpha\|\|\|\alpha\|^{\frac{p-2}{2}}\|d^*\alpha\| \\
&\leq \frac{(p-2)}{2}\int_M\|\alpha\|^{p-2}\|\nabla\|\alpha\|\|^2 + \frac{(p-2)}{2}\int_M\|\alpha\|^{p-2}\|d^*\alpha\|^2.
\end{align*}
Applying these estimates to \eqref{maineq}, we get
\begin{align*}
&\frac{(p-2)+2}{2}\int_M\|\alpha\|^{p-2}\|d\alpha\|^2 + \frac{(p-2)+2}{2}\int_M\|\alpha\|^{p-2}\|d^*\alpha\|^2 \\
&\hspace{0.4 in}\geq \int_M\|\alpha\|^{p-2}\|\nabla \alpha\|^2+\int_M\|\alpha\|^{p-2}(W_k(\alpha),\alpha)\\
&\hspace{0.4 in}\geq \frac{1}{k+1}\int_M \|\alpha\|^{p-2}\|d\alpha\|^2+\frac{1}{n-k+1}\int_M\|\alpha\|^{p-2}\|d^*\alpha\|^2+\int_M\|\alpha\|^{p-2}(W_k(\alpha),\alpha).
\end{align*}
Let 
\begin{equation*}
C := \max\left\{\frac{k}{k+1},\frac{n-k}{n-k+1}\right\}.
\end{equation*}
Using
\begin{align*}
\int_M\|\alpha\|^{p-2}\|d\alpha\|^2 \leq \left( \int_M \|\alpha\|^p\right)^{1-\frac{2}{p}} \left( \int_M\|d\alpha\|^p\right)^{\frac{2}{p}}
\end{align*}
and
\begin{align*}
\int_M\|\alpha\|^{p-2}\|d^*\alpha\|^2 \leq \left( \int_M \|\alpha\|^p\right)^{1-\frac{2}{p}} \left( \int_M\|d^*\alpha\|^p\right)^{\frac{2}{p}},
\end{align*}
we have
\begin{align*}
\left(C+\frac{(p-2)}{2}\right)&\left(\int_M\|\alpha\|^p\right)^{1-\frac{2}{p}}\left[ \left( \int_M\|d\alpha\|^p\right)^{\frac{2}{p}}+ \left( \int_M\|d^*\alpha\|^p\right)^{\frac{2}{p}}\right]\\
& \geq \int_M \|\alpha\|^{p-2}(W_k(\alpha),\alpha).
\end{align*}
For $p\geq 2$, and using the lower bound of the Weitzenb\"ock curvature \eqref{weitzenbocklower}, we have
\begin{align*}
2^{\frac{2}{p}-1}\left(C+\frac{(p-2)}{2}\right)\left(\int_M\|\alpha\|^p\right)^{1-\frac{2}{p}} \left( \int_M\|d\alpha\|^p+\|d^*\alpha\|^p\right)^{\frac{2}{p}} \geq k(n-k)H\int_M \|\alpha\|^p.
\end{align*}
Using the fact that $\int_M \|d\alpha\|^p+\|d^*\alpha\|^p = \lambda \int_M\|\alpha\|^p$ for eigenform $\alpha$, 
we get
\begin{align*}
\lambda^{\frac{2}{p}} \geq \frac{k(n-k)}{2^{\frac{2}{p}-1}\left(C+\frac{(p-2)}{2}\right)}.
\end{align*}

\section{Boundary conditions}\label{sec5}
In this section we briefly discuss the situation of a compact manifold $M$ with nonempty smooth boundary $\dd M$.  Let $n$ denote the unit outer normal vector and let $J:\dd M\to M$ be the inclusion.  Then $J^*\alpha$ is the restriction of a form to the boundary.  Then $d$ and its adjoint $d^*$ are related with an additional boundary term given by
\begin{align*}
\int_M\langle d\alpha,\beta\rangle = \int_M \langle \alpha,d^*\beta\rangle +\int_{\dd M}\langle J^*(\alpha),\iota_n\beta\rangle, \quad \alpha \in \Omega^{k}(M), \beta \in \Omega^{k+1}(M).
\end{align*}
and the corresponding Green's formula for the $p$-Laplacian is 
\begin{align*}
(\Delta_p\alpha,\beta) &= \int_M\|d\alpha\|^{p-2}\langle d\alpha,d\beta\rangle + \int_M\|d^*\alpha\|^{p-2}\langle d^*\alpha,d^*\beta\rangle \\
&\hspace{0.2 in} -\int_{\dd M} \langle \iota_n(\|d\alpha\|^{p-2}d\alpha),J^*(\beta)\rangle +\int_{\dd M} \langle \|d^*\alpha\|^{p-2}J^*(d^*\alpha),\iota_n\beta\rangle.
\end{align*}
The two most common boundary conditions for the classical Laplacian eigenvalue problem  are the Dirichlet and Neumann boundary condition.  For the Hodge-Laplacian, the analogous boundary conditions are the absolute boundary condition
\begin{equation*}
\begin{cases}
\iota_n\alpha = 0 \\
\iota_nd\alpha =0, \quad \text{ on }\dd M
\end{cases}
\end{equation*}
and the relative boundary condition
\begin{equation*}
\begin{cases}
J^*(\alpha) = 0\\
J^*(d^*\alpha) =0, \quad \text{ on }\dd M.
\end{cases}
\end{equation*}
The essential feature of the boundary condition is that if $\alpha$ satisfies either of the boundary conditions, then $\Delta_p\alpha = 0$ implies $d\alpha = 0$ and $d^*\alpha = 0$.  The boundary terms that will be introduced to \eqref{maineq} are
\begin{align*}
& \int_M \langle d(\|\alpha\|^{p-2})\wedge \alpha,d\alpha \rangle - \int_M\langle\iota_{\nabla\|\alpha\|^{p-2}}\alpha,d^*\alpha\rangle +  \int_M \|\alpha\|^{p-2}\|d\alpha\|^2 + \int_M \|\alpha\|^{p-2}\|d^*\alpha\|^2\\
&\hspace{0.2 in}-\int_{\dd M} \|\alpha\|^{p-2}\langle J^*(\alpha),\iota_n(d\alpha)\rangle +\int_{\dd M}\|\alpha\|^{p-2} \langle J^*(d^*\alpha),\iota_n(\alpha)\rangle\\
&=\int_M \left((p-2)\|\alpha\|^{p-2}|\nabla\|\alpha\||^2+\|\alpha\|^{p-2}\|\nabla \alpha\|^2+\|\alpha\|^{p-2}(W_k(\alpha),\alpha)\right).
\end{align*}
Since the boundary terms will vanish under either of the boundary conditions, we get the same estimate for the boundary value problem as well.  It would be interesting to see what the Reilly formula, for instance a generalization of Theorem 3 in \cite{raulot-savo} would be in this context, however due to the asymmetry of the weight function in the $p$-Laplacian, it is not immediate what the appropriate Bochner-Weitzenb\"ock type formula would be for $\Delta_p$.


\begin{bibdiv}
\begin{biblist}
\bib{dung}{article}{
author={Dung, Nguyen Thac},
author={Sung, Chiung Jue Anna },
title={Analysis of weighted $p$-harmonic forms and applications},
journal={(preprint)}
date={2019}
}

\bib{gallot-meyer}{article}{
   author={Gallot, S.},
   author={Meyer, D.},
   title={Sur la premi\`ere valeur propre du $p$-spectre pour les vari\'{e}t\'{e}s \`a
   op\'{e}rateur de courbure positif},
   language={French},
   journal={C. R. Acad. Sci. Paris S\'{e}r. A-B},
   volume={276},
   date={1973},
   pages={A1619--A1621},
   review={\MR{0322735}},
}
\bib{gallot-meyer-2}{article}{
   author={Gallot, S.},
   author={Meyer, D.},
   title={Op\'{e}rateur de courbure et laplacien des formes diff\'{e}rentielles
   d'une vari\'{e}t\'{e} riemannienne},
   language={French},
   journal={J. Math. Pures Appl. (9)},
   volume={54},
   date={1975},
   number={3},
   pages={259--284},
   issn={0021-7824},
   review={\MR{0454884}},
}

\bib{kwong}{article}{
   author={Kwong, Kwok-Kun},
   title={Some sharp Hodge Laplacian and Steklov eigenvalue estimates for
   differential forms},
   journal={Calc. Var. Partial Differential Equations},
   volume={55},
   date={2016},
   number={2},
   pages={Art. 38, 14},
   issn={0944-2669},
   review={\MR{3478292}},
   doi={10.1007/s00526-016-0977-8},
}

\bib{lindqvist}{book}{
   author={Lindqvist, Peter},
   title={Notes on the $p$-Laplace equation},
   series={Report. University of Jyv\"{a}skyl\"{a} Department of Mathematics and
   Statistics},
   volume={102},
   publisher={University of Jyv\"{a}skyl\"{a}, Jyv\"{a}skyl\"{a}},
   date={2006},
   pages={ii+80},
   isbn={951-39-2586-2},
   review={\MR{2242021}},
}

\bib{matei}{article}{
   author={Matei, Ana-Maria},
   title={First eigenvalue for the $p$-Laplace operator},
   journal={Nonlinear Anal.},
   volume={39},
   date={2000},
   number={8, Ser. A: Theory Methods},
   pages={1051--1068},
   issn={0362-546X},
   review={\MR{1735181}},
   doi={10.1016/S0362-546X(98)00266-1},
}

\bib{moroianu-semmelmann}{article}{
   author={Moroianu, Andrei},
   author={Semmelmann, Uwe},
   title={Twistor forms on K\"{a}hler manifolds},
   journal={Ann. Sc. Norm. Super. Pisa Cl. Sci. (5)},
   volume={2},
   date={2003},
   number={4},
   pages={823--845},
   issn={0391-173X},
   review={\MR{2040645}},
}

\bib{naber-valtorta}{article}{
   author={Naber, Aaron},
   author={Valtorta, Daniele},
   title={Sharp estimates on the first eigenvalue of the $p$-Laplacian with
   negative Ricci lower bound},
   journal={Math. Z.},
   volume={277},
   date={2014},
   number={3-4},
   pages={867--891},
   issn={0025-5874},
   review={\MR{3229969}},
   doi={10.1007/s00209-014-1282-x},
}

\bib{raulot-savo}{article}{
   author={Raulot, S.},
   author={Savo, A.},
   title={A Reilly formula and eigenvalue estimates for differential forms},
   journal={J. Geom. Anal.},
   volume={21},
   date={2011},
   number={3},
   pages={620--640},
   issn={1050-6926},
   review={\MR{2810846}},
   doi={10.1007/s12220-010-9161-0},
}

\bib{savo}{article}{
   author={Savo, Alessandro},
   title={On the lowest eigenvalue of the Hodge Laplacian on compact,
   negatively curved domains},
   journal={Ann. Global Anal. Geom.},
   volume={35},
   date={2009},
   number={1},
   pages={39--62},
   issn={0232-704X},
   review={\MR{2480663}},
   doi={10.1007/s10455-008-9121-0},
}

\bib{scott}{article}{
   author={Scott, Chad},
   title={$L^p$ theory of differential forms on manifolds},
   journal={Trans. Amer. Math. Soc.},
   volume={347},
   date={1995},
   number={6},
   pages={2075--2096},
   issn={0002-9947},
   review={\MR{1297538}},
   doi={10.2307/2154923},
}

\bib{seto-wei}{article}{
   author={Seto, Shoo},
   author={Wei, Guofang},
   title={First eigenvalue of the $p$-Laplacian under integral curvature
   condition},
   journal={Nonlinear Anal.},
   volume={163},
   date={2017},
   pages={60--70},
   issn={0362-546X},
   review={\MR{3695968}},
   doi={10.1016/j.na.2017.07.007},
}
\end{biblist}
\end{bibdiv}
\end{document}